\documentclass{elsarticle}

\usepackage[utf8]{inputenc}
\usepackage[english]{babel}
\usepackage{amsmath,amsthm}
\usepackage{amsfonts}
\usepackage{amssymb}

\usepackage{esint}  
\usepackage{enumerate}  
\usepackage{hyperref}  
\hypersetup{colorlinks=true,linktoc=none, linkcolor=blue,citecolor=red}
\usepackage{etoolbox}

\theoremstyle{plain}
\newtheorem{theorem}{Theorem}[section]  

\newtheorem{lemma}[theorem]{Lemma}
\theoremstyle{definition}
\newtheorem{definition}[theorem]{Definition}
\theoremstyle{remark}
\newtheorem{remark}[theorem]{Remark}

\newcommand{\e}{\mathrm{e}}  
\newcommand{\A}{\mathcal{A}}  
\newcommand{\Sch}{\mathcal{S}}  
\newcommand{\h}{\mathbb{H}}  
\newcommand{\Ha}{\mathcal{H}}  
\newcommand{\hb}{\widehat{\mathbb{H}}}  
\newcommand{\nablah}{\nabla_{\mathbb{H}}}
\newcommand{\R}{\mathbb{R}}  
\newcommand{\Rfp}{\mathcal{R}_+^{4}}  
\newcommand{\di}{\mathrm{d}}  

\DeclareMathOperator{\divr}{div}  
\DeclareMathOperator{\PV}{P.V.}  
\DeclareMathOperator{\maximum}{max}  

\numberwithin{equation}{section}

\makeatletter
\def\ps@pprintTitle{%
 \let\@oddhead\@empty
 \let\@evenhead\@empty
 \def\@oddfoot{}%
 \let\@evenfoot\@oddfoot}
\makeatother

\begin{document}

\title{Rigidity results for non local phase transitions in the Heisenberg group $\h$\footnote{NOTICE: this is the author's version of a work that was accepted for publication in {\em Discrete and Continuous Dynamical Systems -- Series A (DCDS-A)}. Changes resulting from the publishing process, such as peer review, editing, corrections, structural formatting, and other quality control mechanisms may not be reflected in this document. Changes may have been made to this work since it was submitted for publication.}}

\author[LL]{Luis F. L\'{o}pez}
\address[LL]{Departamento de Ingenier\'{\i}a Matem\'atica, Universidad
de Chile, Casilla 170 Correo 3, Santiago, Chile}
\ead{llopez@dim.uchile.cl}

\author[YS]{Yannick Sire} 
\address[YS]{Aix-Marseille Universit\'e,
Laboratoire d'Analyse, Topologie, Probabilit\'e (LATP),
9, rue F. Joliot Curie,
13453 Marseille Cedex 13 France}
\ead{sire@cmi.univ-mrs.fr}

\begin{abstract}
In the Heisenberg group framework, we study rigidity properties for stable solutions of $(-\Delta_\h)^sv=f(v)$ in $\h$, $s\in(0,1)$. We obtain a Poincar\'e type inequality in connection with a degenerate elliptic equation in $\R^4_+$; through an extension (or ``lifting") procedure, this inequality will be then used for giving a criterion under which the level sets of the above solutions are minimal surfaces in $\h$, i.e. they have vanishing mean curvature.    
\end{abstract}


\maketitle



\section{Introduction}

In this paper we study rigidity properties for stable solutions (see Definition~\ref{def1}) of non-local equations of the type 
\begin{equation}\label{1-1}
(-\Delta_\h)^sv=f(v)\quad\text{in}\ \h,
\end{equation}
where $s\in(0,1)$, $f\in C^{1,\gamma}(\R)$, $\gamma>\maximum \{0,1-2s\}$ and $\h$ is the Heisenberg group (see Section~2).

We want to give a geometric insight of the phase transition for equation~\eqref{1-1}. Following the ideas in \cite{SiVa}, we give a geometric proof of rigidity properties for fractional boundary reactions in $\h$. 

The relation between entire stable solutions and minimal surfaces, as performed in Theorem~\ref{thm2} and Theorem~\ref{thm3} below, is inspired by a famous conjecture of De Giorgi \cite{DeGi} (in the Euclidean setting) and in the spirit of the proof of Bernstein theorem given in \cite{Gi84}. Similar De Giorgi-type results (in the Euclidean setting) have been proven in \cite{CaSo} for the square root of the Laplacian, and later generalized in \cite{CaSi07} for arbitrary roots. In \cite{SiVa} is given a proof of analogous rigidity properties for phase transitions driven by fractional Laplacians. Unlike the method in \cite{CaSi07,CaSo}, which require a Liouville-type result, the proof in \cite{SiVa} is based on the recent work \cite{FaScVa} and relies heavily on a Poincar\'e-type inequality which involves the geometry of the level sets of $u$.       

To find such a Poincar\'e-type inequality, we shall use a suitably development of some techniques for level set analysis inspired by \cite{Fa,FaScVa,SiVa,StZu1,StZu2}; some properly modified computations, contained in \cite{FeVa08}, are needed in order to understand the complicated geometry of the Heisenberg group. As a result, we find a Poincaré-type inequality for stable solutions of a degenerate elliptic equation in $\R^4_+$ (see \eqref{1-2}). We use this inequality together with an ``abstract" formulation of a technique recently introduced by Caffarelli and Silvestre \cite{CaSi}, to study \eqref{1-1}.  

In Euclidean spaces, fractional operators have been studied in connection with different phenomena such as optimization \cite{DuLi}, flame propagation \cite{CaRoSi} and finance \cite{CoTa}. We also mention the thin obstacle problem and phase transition problems: see, for instance, \cite{CaSaSi} and \cite{SiVa}. From a probabilistic point of view, the standard fractional Laplacian is the infinitesimal generator of a Levy process (see, e.g., \cite{Be}). 

The standard fractional Laplacian is a non-local operator. This fact does not allow to apply local PDE techniques to treat nonlinear problems for $(-\Delta)^s$. To overcome this difficulty, Caffarelli and Silvestre showed in \cite{CaSi} that any fractional power of the Laplacian can be determined as an operator that maps a Dirichlet boundary condition to a Neumann-type condition via an extension problem. More precisely, let us consider the boundary reaction problem for $u=u(x,y)$, $x\in\R^N$ and $y>0$,
\begin{equation}\label{1-1a}
\left\{
 \begin{aligned}
 \divr(y^a &\nabla u)=0 && \text{in}\ \R^N\times(0,\infty),\\
 -y^au_y &=f(u) && \text{on}\ \R^N\times\{0\},
 \end{aligned}
\right.
\end{equation} 
where $a=1-2s$. It is proved in \cite{CaSi} that, up to a normalizing factor, the Dirichlet-to-Neumann operator $\Gamma_a:u|_{\partial\R^{N+1}_+}\mapsto-y^a u_y|_{\partial\R^{N+1}_+}$ is precisely $(-\Delta)^s$ and then that $u(x,0)$ is a solution of
\begin{equation}\label{1-1b}
(-\Delta)^su(x,0)=f(u(x,0)).
\end{equation}

On the other hand, sub-Laplacians in Carnot groups (i.e. simply connected stratified nilpotent Lie groups) exhibit strong analogies with classical Laplace operators in the Euclidean space (Harnack inequality, maximum principle, existence and estimates of the fundamental solution). Following \cite{CaSi}, a construction of a $\Delta_{\h}$-harmonic ``lifting'' operator $v=v(x)\mapsto u=u(x,y)$ from $\h$ to $\h\times\R^+$ can be carried out by means of the spectral resolution of $-\Delta_{\h}$ in $L^2(\h)$ in such a way that $v$ is the trace of the normal derivative of $u$ on $\{y=0\}$ (see \cite{FeFr} and the references therein). 

For the time being, we leave the precise framework for Section~2, instead we discuss the main results.

Let us define $\hb:=\h\times(0,+\infty)$. As in the Euclidean case, the study of the non-local equation \eqref{1-1} is related to the analysis of the following degenerate elliptic problem (see Section~2 for details):
\begin{equation}\label{1-2}
\left\{
 \begin{aligned}
 \divr_{\hb}(y^a & \nabla_{\hb}u)=0 &&\text{in } \h \times(0,\infty),\\
 -y^au_y&=f(u) &&\text{on } \h \times\{0\}.
 \end{aligned}
\right.
\end{equation}

\begin{definition}[Functional framework] 
\begin{enumerate}[(I)]
\item {\em Notion of weak solution}: \eqref{1-2} may be understood in the {\em weak sense}, namely supposing that $u\in L^\infty_{\mathrm{loc}}(\overline{\R^4_+})$ with
\begin{equation}\label{regularity}
y^a|\nabla_{\hb}u|^2\in L^1(B_R^+)
\end{equation}
for any $R>0$, and that
\begin{equation}\label{1-3}
\int_{\hb}y^a\langle\nabla_{\hb}u,\nabla_{\hb}\xi\rangle_{\hb}=\int_{\h}f(u)\xi
\end{equation}
for all $\xi:\R^4_+\rightarrow\R$ bounded, locally Lipschitz, which vanishes on $\R_+^4\setminus B_R$ and such that
\begin{equation}\label{1-4}
y^a|\nabla_{\hb}\xi|^2\in L^1(B_R^+).
\end{equation}
We use here the notation $\R_+^4=\R^3\times(0,\infty)$ and $B_R^+:=B_R\cap\R_4^+$.

\item {\em Notion of stability}: Let $u$ be a weak solution of \eqref{1-2}, $u$ is {\em stable} if  
\begin{equation}\label{1-5}
\int_{\hb}y^a|\nabla_{\hb}\xi|^2-\int_{\h}f'(u)\xi^2\geq 0
\end{equation}
for any $\xi$ as above. This condition is natural in the calculus of variation framework, in particular it says that the second variation of the associated functional has a sign, as it happens for local minima, for instance.
\end{enumerate}
\end{definition}

For the precise statement of our geometric result, we introduce the following notation: fixed $y>0$ and $c\in\R$, we look at the level set 
\[S:=\{x\in\R^3\ \text{s.t.}\ u(x,y)=c\},\]
and we consider the regular points of $S$, i.e.
\begin{equation}\label{1-6}
L:=\{x\in S\ \text{s.t.}\ \nablah u(x,y)\neq 0\}.
\end{equation}
Although $S$ and $L$ depend on $y\in(0,+\infty)$, we do not make it explicit in the notation.

We also define 
\[
 \Rfp:=\{(x,y)\in\h\times(0,+\infty)\ \text{s.t.}\ \nablah u(x,y)\neq 0\}.
\]
Since L is a smooth manifold, given $x\in L$, we denote:
\begin{enumerate}[(i)]
\item $\nu_{x,y}$ be the {\em intrinsic normal} along $L$,
\item $v_{x,y}$ be the {\em intrinsic unit tangent} along $L$,
\item $h_{x,y}$ be the {\em the intrinsic mean curvature} along $L$,
\item $p_{x,y}$ be the {\em imaginary curvature} along $L$,
\end{enumerate}
(see Definition~\ref{def2} for details).

In this framework, we can state our geometric formula; see Section~2 for the definition of the vector fields $X,Y,T$ and the Hessian matrix $H$.

\begin{theorem}\label{thm1}
Let $u\in C^2(\hb)$ be a bounded and stable weak solution of \eqref{1-2}. Assume furthermore that for all $R>0$,
 \begin{equation}\label{gradient_bound}
 |\nablah u|\in L^\infty(\overline{B_R^+}). 
 \end{equation}
 Then, for any $\phi\in C_0^\infty(\R^4)$, we have 
 \begin{equation}\label{1-7}
 \begin{aligned}
  \int\limits_{\hb}&y^a|\nablah u|^2|\nabla_{\hb}\phi|^2\\
   &\geq \int\limits_{\Rfp}y^a\left(|Hu|^2-\langle(Hu)^2\nu_{x,y},\nu_{x,y}\rangle_{\h}-2(TYuXu-TXuYu)\right)\phi^2\\
   &= \int\limits_{\Rfp}y^a|\nablah u|^2\left[h_{x,y}^2+\left(p_{x,y}+\frac{\langle Huv_{x,y},\nu_{x,y}\rangle_{\h}}
   {|\nablah  u|}\right)^2+2\langle T\nu_{x,y},v_{x,y}\rangle_{\h}\right]\phi^2.
 \end{aligned} 
\end{equation}
\end{theorem}

\begin{remark}
We observe that \eqref{1-7} may be interpreted in two ways:
\begin{enumerate}[(i)]
\item One way is to think that some interesting geometric objects which describe $u$, such as its intrinsic Hessian and the curvature of its level sets, are bounded by an energy term. These quantities involved in the inequality are weighted by a test function $\phi$ which can be chosen as we wish.
\item Another point of view consists in thinking that \eqref{1-7} bounds a suitably weighted $L^2$-norm of its gradient. The weights here are given by the stable solution $u$. So, this interpretation sees \eqref{1-7} as a Sobolev-Poincaré inequality.
\end{enumerate}
\end{remark}
The result in Theorem~\ref{thm1} has been inspired by \cite{StZu1,StZu2}; in particular, they obtained a similar inequality for stable solutions of the Allen-Cahn equation, and symmetry results for possibly singular or degenerated models have been obtained in \cite{Fa,FaScVa}. Actually, the study of geometric inequalities for semilinear equations goes back to \cite{StZu1,StZu2}, where uniformly elliptic PDEs in the Euclidean space were taken into account, and further important developments have been performed in \cite{Fa}.  Recently, in \cite{SiVa} has been proved a similar inequality to \eqref{1-7} in the Euclidean setting. Related geometric inequalities also played an important role in \cite{CaCap}. 

The next theorem is a rigidity result. For the precise statement of it, let us define the following suitably weighted energy:

\begin{equation}\label{1-8}
\eta(\tau)=\int_{B(0,\tau)}\! 4y^a|\nablah u(x_1,x_2,x_3,y)|^2(x_1^2+x_2^2+y^2)\,\di(x_1,x_2,x_3,y).
\end{equation}
In this expression, $B(0,\tau)$ represents a ball in $\hb$ with a gauge norm that will be defined in Section~4 (see \eqref{4-1} and \eqref{4-2}). No confusion should arise with the Euclidean ball.  

\begin{theorem}\label{thm2}
 Let the assumptions of the previous theorem hold. Suppose also that
 \begin{equation}\label{1-9}
 \langle T\nu_{x,y},v_{x,y}\rangle_\h\geq 0 \quad \text{for all } x\in\h, y>0;
 \end{equation}
 and $\eta$, previously defined, satisfies the growth
 \begin{equation}\label{1-10}
 \liminf_{R\to+\infty}\frac{\int_{\sqrt{R}}^R\frac{\eta(\tau)}{\tau^5}\,\di\tau+\frac{\eta(R)}{R^4}}{\log^2R}=0.
 \end{equation}
 Then, the level sets of $u$ intersected with $L$ (recall \eqref{1-6}) are minimal surfaces in the Heisenberg group (i.e., the  
 curvature $h_{x,y}$ vanishes identically) and on such surfaces the following holds
 \begin{equation}
 p_{x,y}=-\frac{\langle Huv_{x,y},\nu_{x,y}\rangle_{\h}}
 {|\nablah  u|}.
 \end{equation}
\end{theorem}

\begin{remark}
 \begin{enumerate}[(i)]
 \item We observe that \eqref{1-10} may be seen as a condition on the growth of a suitably weighted energy $\eta$.
 \item Notice also that if, for any $R$ large enough, \[\eta(R)\leq CR^4,\] for some constant $C>0$, then \eqref{1-10} is satisfied.
 \end{enumerate} 
\end{remark}

Before stating the rigidity result, let us precise the notion of stable solution for equation~\eqref{1-1}:
\begin{definition}\label{def1}
A bounded solution $v\in C^2(\h)$ of \eqref{1-1} is stable if for all $\varphi\in W_{\h}^{s,2}(\h)$ we have
\begin{equation}
\int_{\h}|(-\Delta_{\h})^{\frac{s}{2}}\varphi|^2-\int_{\h}f'(v)\varphi^2\geq 0.
\end{equation}
For the precise definition of the space $ W_{\h}^{s,2}(\h)$ and the fractional operator $(-\Delta_{\h})^{\frac{s}{2}}$, we refer the reader to Section~2.
\end{definition}

Throughout the paper, $C^\alpha(\h)$ denotes the set of H\"older continuous functions with respect to the norm $\rho$ defined in the next section (see \eqref{rho-norm}). Our rigidity result is the following:
\begin{theorem}\label{thm3}
Let $v\in C^{2,\sigma}(\h)$, $\sigma\in(0,2s)$, be a bounded stable solution of \eqref{1-1}. Assume also that
the ``harmonic lifting" of $v$ to $\hb$ (see Subsection~2.3), which we denote by $u$, satisfies \eqref{1-9} and \eqref{1-10}. Then, the level sets of $v$ in the vicinity of non-characteristic points are minimal surfaces in the Heisenberg group (i.e., the curvature $h$ vanishes identically).
\end{theorem}

The paper is organized as follows: In Section~2 we recall the definition and the basic properties of the Heisenberg group, as well as the precise definition of the fractional sub-Laplacian involved in Eq.~\eqref{1-1}; we also discuss some regularity properties related to the degenerate elliptic problem \eqref{1-2}. In Section~3 we shall develop the analytical tools toward \eqref{1-7}, in particular one part of this inequality will be given in Theorem~\ref{thm6}; the geometry of the Heisenberg group will be fundamental in the proof of Theorem~\ref{thm1} at the end of Section~3 (see \cite[Section~2]{FeVa08}). Finally, Section~4 contains the application to the stable solutions in the entire space; we prove Theorem~\ref{thm2} and Theorem~\ref{thm3}.


\section{Preliminaries}

Let us briefly recall the definition and the basic properties of the Heisenberg group, so we will be able to precise the meaning of the fractional sub-Laplacian operator involved in \eqref{1-1}.

\subsection{The Heisenberg group}
Let $\h$ be the Heisenberg group, namely $\R^3$ endowed with the following non-commutative law: for every $(x_1,x_2,x_3),\ (y_1,y_2,y_3)\in\R^3$
\[(x_1,x_2,x_3)\circ(y_1,y_2,y_3)=(x_1+y_1,x_2+y_2,x_3+y_3+2(x_2y_1-x_1y_2)).\]
We shall denote $X=(1,0,2x_2)$ and $Y=(0,1,-2x_1)$. With the same notation we denote the two vector fields $X=\frac{\partial}{\partial x_1}+2x_2\frac{\partial}{x_3}$ and $Y=\frac{\partial}{\partial x_2}-2x_1\frac{\partial}{x_3}$ generating the algebra. We denote also by 
\[T:=[X,Y]=-4\frac{\partial}{\partial x_3}.\]
In particular, on each fiber $\mathcal{H}_P=\textrm{span}\{X,Y\}$ an internal product is given as follows: for every $U,\ V\in\mathcal{H}_P$, with $U=\alpha_1X+\beta_1Y$ and $V=\alpha_2X+\beta_2Y$, we have
\[
 \langle U,V\rangle_{\h}=\alpha_1\alpha_2+\beta_1\beta_2.
\]
This internal product makes the vectors $X$ and $Y$ orthonormal on $\mathcal{H}_P$. We shall denote the norm on $\mathcal{H}_P$ for every $U\in\mathcal{H}_P$ as 
\[|U|_{\h}=\sqrt{\langle U,X\rangle_{\h}^2+\langle U,Y\rangle_{\h}^2}.\]
No confusion should arise between the Euclidean objects $\langle\cdot,\cdot\rangle$ and $|\cdot|$ and the ones on the fibers in the Heisenberg group respectively denoted by $\langle\cdot,\cdot\rangle_{\h}$ and $|\cdot|_{\h}$.

For a smooth function $u$, we denote $\nablah u(P)=(Xu(P),Yu(P))$, where $Xu(P)$ and $Yu(P)$ are the coordinates of the vector $\nablah u(P)$ with respect to the basis given by $X$ and $Y$ at $P$. The vector $\nablah u$ is called the {\em intrinsic gradient} of $u$.

\begin{definition}\label{def2}
 \begin{enumerate}[(I)]
 \item We remind that a point $P\in\Sigma$ is {\em characteristic} for the $C^1$ level set $\Sigma$ of $u$ when the fiber in $P$  
 coincides with the Euclidean tangent space $\Sigma$ at $P$, namely $\mathcal{H}_P=T_P\Sigma$. In particular if $\nablah 
 u(P)\neq 0$, then $P$ is not characteristic. 
 \item Whenever $P\in\{u=k\}\cap\{\nablah u\neq 0\}$, one can consider the smooth surface $\{u=k\}$ and define 
 \[\nu= \frac{\nablah u(P)}{|\nablah u(P)|}.\]
 Usually, such $\nu$ is called the {\em intrinsic normal}. Associated with $\nu$, to any non-characteristic point $P\in\{u=k\}$,  
 there exists the so called {\em intrinsic unit tangent} direction to the level set $\{u=k\}$ at $P$ defined as 
 \[v=\frac{(Yu(P),-Xu(P))}{|\nablah u|},\]
 where the above coordinates are given with respect to the $(X,Y)$-frame. We observe that $\nu$ and $v$ are orthonormal in $\h$.
 \item The {\em intrinsic mean curvature} $h$, in a non-characteristic point $P\in\Sigma$ of the level surface given by $u$, is  
 defined as
 \[h=\mathrm{div}_{\h}\nu(P),\]
 while the {\em imaginary curvature} $p$ at the point $P\in\Sigma$ of the level surface $\Sigma$, given by $u$, is defined as  
 \[p=-\frac{Tu(P)}{|\nablah u(P)|}.\]
 \end{enumerate}
\end{definition}
\begin{remark}
For the notion of {\em intrinsic mean curvature} we refer to \cite{ArFe07,ArFe08,CDPT,ChHwMa,GaDaNh,Pa04}, while for the notion of {\em imaginary curvature} and its geometric meaning we refer to \cite{ArFe07,ArFe08}.
\end{remark}

The Kohn-Laplace operator on $\h$ is defined by \[\Delta_{\h}u=X^2u+Y^2u.\] Since a divergence operator is defined on each fiber, we can write
\[
 \Delta_{\h}u=\divr_{\h}(\nablah u)=X(Xu)+Y(Yu).
\]
With regards to problem \eqref{1-2}, we define $\hb:=\h\times\R_+$ and given $u$ and $h=(h_1,h_2,h_3)$ we denote
\[
\nabla_{\hb}u=(Xu,Yu,u_y),\quad\divr_{\hb}h=Xh_1+Yh_2+\partial_yh_3.
\]

We define the horizontal intrinsic Hessian matrix as 
\begin{equation*}
Hu=
 \begin{bmatrix}
 XXu&YXu\\
 XYu&YYu
 \end{bmatrix}.
\end{equation*}
Its norm is given by \[|Hu|=\sqrt{(XXu)^2+(YXu)^2+(XYu)^2+(YYu)^2}.\]
As usual, we set \[(Hu)^2=(Hu)(Hu)^T.\]

For any $\lambda>0$, the dilatation $\delta_\lambda:\h\rightarrow\h$ is defined as
\begin{equation}\label{2-1a}
\delta_\lambda(x_1,x_2,x_3)=(\lambda x_1,\lambda x_2,\lambda^2 x_3).
\end{equation}
Through this paper, by $\h$-homogeneity we mean homogeneity with respect to group dilatations $\delta_\lambda$.

The Haar measure of $\h=(\R^3,\cdot)$ is the Lebesgue measure $\mathcal{L}^3$ in $\R^3$. If $A\subset\h$ is $\mathcal{L}^3$-measurable, we write also $|A|:=\mathcal{L}^3(A)$. Moreover, if $m\geq 0$, we denote by $\Ha^m$ the m-dimensional Hausdorff measure obtained from the Euclidean distance in $\R^3\simeq\h$. 

We refer the reader to \cite[Chapter~5]{BoLaUg} for the definition of the Carnot-Carath\'eodory in $\h$ (cc-distance) $d_c(x,y)$. We shall denote $B_c(x,r)$ the open balls associated with $d_c$. The cc-distance is well behaved with respect to left translations and dilatations, that is
\[
d_c(z\circ x,z\circ y)=d_c(x,y), \quad d_c(\delta_\lambda(x),\delta_\lambda(y))=\lambda d_c(x,y)
\]
for $x,y,z\in\h$ and $\lambda>0$.

We also have 
\begin{equation}\label{2-1b}
|B_c(x,r)|=r^4|B_c(0,1)| \quad \text{and}\quad |\partial B_c(x,r)|=r^3|\partial B_c(0,1)|
\end{equation}
(recall that $4=$ homogeneous dimension of $\h$).

We can define a group convolution in $\h$: if, for instance, $f\in\mathcal{D}(\h)$ and $g\in L^1_{\mathrm{loc}}(\h)$, we set 
\begin{equation}\label{2-2}
f\ast g(x):=\int_\h \! f(y)g(y^{-1}\circ x)\,\di y \quad \text{for } x\in\h
\end{equation}
(here $y^{-1}$ denotes the inverse in $\h$). We remind that the convolution is well defined when $f,g\in\mathcal{D}'(\h)$, provided at least one of them has compact support.

\subsection{Fractional powers of sub-elliptic Laplacians}

Here, we collect some results on fractional powers of sub-Laplacian in the Heisenberg group (see \cite{FeFr,Fo75}).

To begin with, let us characterize $(-\Delta_\h)^s$ as the spectral resolution of $\Delta_\h$ in $L^2(\h)$ (see \cite[Theorem~3.10]{FeFr} and \cite[Section~3]{Fo75}).

\begin{theorem}
The operator $\Delta_\h$ is a positive self-adjoint operator with domain $W_\h^{2,2}(\h)$. Denote now by $\{E(\lambda)\}$ the spectral resolution of $\Delta_\h$ in $L^2(\h)$. If $\alpha>0$ then
\[
(-\Delta_\h)^{\alpha/2}=\int_0^{+\infty}\! \lambda^{\alpha/2} \,\di E(\lambda)
\]
with domain
\[
W_{\h}^{\alpha,2}(\h):=\{v\in L^2(\h):\ \int_0^{+\infty}\! \lambda^\alpha\di \langle E(\lambda)v,v\rangle<\infty\},
\]
endowed with the graph norm.
\end{theorem}

Before giving a more ``explicit'' expression of the fractional sub-Laplacian, we recall some definitions. Denote by $h=h(t,x)$ the fundamental solution of $\Delta_\h+\partial/\partial t$ (see \cite[Proposition~3.3]{Fo75}). For all $0<\beta<4$ the integral 
\[
R_\beta(x)=\frac{1}{\Gamma(\beta/2)}\int_0^{+\infty}\! t^{\frac{\beta}{2}-1}h(t,x)\, \di t
\]
converges absolutely for $x\neq 0$. 

If $\beta<0, \beta\notin\{0,-2.-4,...\}$, then
\[
\widetilde{R}_\beta(x)=\frac{\frac{\beta}{2}}{\Gamma(\beta/2)}\int_0^{+\infty}\! t^{\frac{\beta}{2}-1}h(t,x)\, \di t
\]
defines a smooth function in $\h\setminus\{0\}$, since $t\mapsto h(t,x)$ vanishes of infinite order as $t\to 0$ if $x\neq 0$. In addition, $\widetilde{R}_{\beta}$ is positive and $\h$-homogeneous of degree $\beta-4$.

We also set
\begin{equation}\label{rho-norm}
\rho(x)=R^{-\frac{1}{2+\alpha}}_{2-\alpha}(x),\quad 0<\alpha<2.
\end{equation}
$\rho$ is an $\h$-homogeneous norm in $\h$, smooth outside of the origin. In addition, $d(x,y):=\rho(y^{-1}\circ x)$ is a quasi-distance in $\h$. In turn, $d$ is equivalent to the Carnot-Carath\'eodoty distance on $\h$, as well as to any other $\h$-homogeneous left invariant distance on $\h$.

Recall that, as usual, $\Sch$ denotes the Schwartz space of rapidly decreasing $C^\infty$ functions. We have the following representation formula:
\begin{theorem}[\cite{FeFr}, Theorem~3.11]\label{thm4}
For every $v\in\Sch(\h)$, $(-\Delta_\h)^sv\in L^2(\h)$ and
\begin{equation*}
 \begin{aligned}
 (-\Delta_\h)^sv(x)&=\int_{\h}\! \left(v(x\circ y)-v(x)-\omega(y)\langle\nablah v(x),y\rangle\right)\widetilde{R}_{-2s}(y)\,\di y\\
 &=\PV \int_{\h}\! (v(y)-v(x))\widetilde{R}_{-2s}(y^{-1}\circ x)\, \di y,
 \end{aligned}
\end{equation*}
where $\omega$ is the characteristic function of the unit ball $B_\rho(0,1)$.
\end{theorem}

\subsection{A Poisson Kernel}
With a natural notion of group convolution, the Heisenberg group makes possible to recover, starting from the abstract representation in terms of spectral resolution, another explicit form of the fractional power in terms of the convolution with suitable Poisson kernel (see \cite[Theorem~4.4]{FeFr}). 

If $v\in L^2(\h)$ and $y>0$ (recall that $-1<a<1$), we set
\[
u(\cdot,y):=\phi(\theta y^{1-a}(-\Delta_\h)^{(1-a)/2})v:=\int_0^{+\infty}\! \phi(\theta y^{1-a}\lambda^{(1-a)/2})\, \di E(\lambda)v, 
\]
where $\theta:=(1-a)^{a-1}$ and $\phi:[0,\infty)\rightarrow\R$ solves the boundary value problem
\begin{equation*}
 \left\{
  \begin{aligned}
  -t^\alpha\phi''+\phi&=0,\\
  \phi(0)&=1,\\
  \lim_{t\to+\infty}\phi(t)&=0,
  \end{aligned}
 \right.
\end{equation*}
($\alpha=-\frac{2a}{1-a}$).

We denote by $h(t,\cdot)$ the heat kernel associated with $-\Delta_\h$ as in \cite[Proposition~3.3]{Fo75}, and by $P_{\h}(\cdot,y)$ the ``Poisson kernel''
\begin{equation}\label{2-3}
P_\h(\cdot,y):= C_a y^{1-a}\int_0^\infty \! t^{(a-3)/2}\e^{-\frac{y^2}{4t}}h(t,\cdot)\,\di t,
\end{equation}
where
\[
C_a=\frac{2^{a-1}}{\Gamma((1-a)/2)}.
\]
Then
\[
 P_\h(\cdot,y)\geq 0
\]
and
\begin{equation}\label{2-4}
u(\cdot,y)=v\ast P_\h(\cdot,y).
\end{equation}

\begin{remark}\label{rmk1}
We note that $v\in C^2(\R^3)\cap L^\infty(\R^3)$ is a stable solution of \eqref{1-1} if and only if its lifting $u(\cdot,y)=v\ast P_\h(\cdot,y)$ is a stable solution of \eqref{1-2}.
\end{remark}

\subsection{Regularity theory for \eqref{1-1} and \eqref{1-2}}

Some classical pointwise estimates: the Harnack inequality and the H\"older continuity of the weak solutions (De Giorgi-Nash-Moser theorem), can be extended to a class of strongly degenerate elliptic operators of the second order, like that in \eqref{1-2}, see \cite{FeFr,Lu92}. 

We state a result which let us control further derivatives in $x$. Basically, this is possible thanks to the fact that the operator is independent of the variable $x\in\h$. 

\begin{lemma}\label{lemma1}
Let $u$ be a bounded weak solution of \eqref{1-2}. Then,
\[
y^a|\nabla_{\hb} Xu|^2,y^a|\nabla_{\hb} Yu|^2\in L^1(B_R^+) 
\]
for every $R>0$.
\end{lemma}
\begin{proof}
Given $R>0$, let us prove that $y^a|\nabla_{\hb} Xu|^2\in L^1(B_R^+)$. 

We consider the incremental quotient
\[
u_h(x,y)=\frac{u(x\circ(he_1),y)-u(x,y)}{h} \quad\text{for all }(x,y)\in\hb,
\]
where $e_1=(1,0,0)$. Recall that (see Proposition~1.2.11 in \cite{BoLaUg}) 
\begin{equation}\label{2-5}
\lim_{h\to 0}u_h(x,y)=Xu(x,y) \quad \text{for all }(x,y)\in\hb.
\end{equation}
Thanks to \eqref{gradient_bound} and the smoothness of $f$, we have
\begin{equation}\label{2-6}
[f(u)]_h\leq C
\end{equation}
for some $C>0$.

Let now $\xi$ be as requested in \eqref{1-3}. We have
\begin{align*}
\int_{\hb}y^a\langle\nabla_{\hb}u_h,\nabla_{\hb}\xi\rangle_{\hb}-\int_{\h}[f(u)]_h\xi&=-\int_{\hb}y^a\langle\nabla_{\hb}u,\nabla_{\hb}\xi_{-h}\rangle_{\hb}+\int_{\h}f(u)\xi_{-h}\\
&=0.
\end{align*}
We now consider a smooth cut-off function $\tau$ such that $0\leq\tau\in C_0^\infty(B_{R+1})$, with $\tau=1$ in $B_R$ and $\nabla\tau\leq 2$. Taking $\xi:=u_h\tau^2$ in the above expression, we find that
\begin{equation}\label{2-7}
2\int_{\hb}y^a\tau u_h\langle\nabla_{\hb}u_h,\nabla_{\hb}\tau\rangle_{\hb}+\int_{\hb}y^a\tau^2|\nabla_{\hb}u_h|^2=\int_{\h}[f(u)]_hu_h\tau^2.
\end{equation}
Note that $\xi$ satisfies \eqref{1-4} thanks to \eqref{regularity} and $u\in L^\infty_{\mathrm{loc}}(\overline{\R^4_+})$.

Now, by Cauchy-Schwarz inequality, we have
\begin{align*}
\int_{\hb}y^a\tau u_h\langle\nabla_{\hb}u_h,\nabla_{\hb}\tau\rangle_{\hb} \geq & -\frac{\epsilon}{2}\int_{\hb}y^a\tau^2|\nabla_{\hb}u_h|^2\\
&-\frac{1}{2\epsilon}\int_{\hb}y^au_h^2|\nabla_{\hb}\tau|^2
\end{align*}
for any $\epsilon>0$. Choosing $\epsilon$ small, \eqref{2-7} reads
\begin{equation}
\int_{\hb}y^a\tau^2|\nabla_{\hb}u_h|^2\leq C\left(\int_{B_{R+1}^+}\!y^au_h^2+\int_{\{|x|\leq R\}\times\{y=0\}}\!|[f(u)]_hu_h|\right)
\end{equation} 
for some $C>0$. This inequality, together \eqref{gradient_bound} and \eqref{2-6}, allows to control 
\[
\int_{\hb}y^a\tau^2|\nabla_{\hb}u_h|^2
\]
uniformly in $h$.

By sending $h\to 0$ and using Fatou lemma (recall also \eqref{2-5}), we obtain the desired claim.
\end{proof}


\section{Analytic and geometric inequalities}

In this section we develop the analytical and geometrical tools toward \eqref{1-7}, we follow the ideas in \cite{FeVa08,SiVa}. We summarize the main points of the argument and omit some technical computations.

\subsection{Analytical computations}

We start with two lemmas, the fist one is a version in the Heisenberg group of a classical result (see \cite{FeVa08}):

\begin{lemma}\label{lemma2}
Let $c\in\R$. Suppose that $\Omega$ is an open domain of $\hb$ and that $w:\Omega\rightarrow\R$ is Lipschitz with respect to the metric structure of $\hb$. Then, $\nabla_{\hb} w=0$ for almost any $x\in\{w=c\}$.
\end{lemma}

And the second one, an elementary observation.

\begin{lemma}\label{lemma3}
Let $u$ be as in Theorem~\ref{thm1}. Assume that $\xi\in C^\infty(\R^4_+,\R)$ and vanishes outside a ball. Then
 \begin{equation}\label{3-1}
 \int_{\hb}y^a\langle\nabla_{\hb}u,\nabla_{\hb}X\xi\rangle_{\hb}=\int_{\hb}
 y^a(-\langle\nabla_{\hb}Xu,\nabla_{\hb}\xi\rangle_{\hb}+2TYu\xi)
 \end{equation}
and
 \begin{equation}\label{3-2}
 \int_{\hb}y^a\langle\nabla_{\hb}u,\nabla_{\hb}Y\xi\rangle_{\hb}=\int_{\hb}
 y^a(-\langle\nabla_{\hb}Yu,\nabla_{\hb}\xi\rangle_{\hb}-2TXu\xi).
 \end{equation}
\end{lemma}

\begin{proof}
Using integration by parts we deduce that
 \begin{align*}
 \int_{\hb}y^a\langle\nabla&_{\hb}u,\nabla_{\hb}X\xi\rangle_{\hb}=\int_{\hb}
 y^a (XuXX\xi+YuYX\xi+\partial_yu\partial_y(X\xi))\\
 &=\int_{\hb}y^a(-XXuX\xi+YuYX\xi-\partial_yXu\partial_y\xi)\\
 &=\int_{\hb}y^a(-\langle\nabla_{\hb}Xu,\nabla_{\hb}\xi\rangle_{\hb}+YXuY\xi+YuYX\xi)\\
 &=\int_{\hb}y^a(-\langle\nabla_{\hb}Xu,\nabla_{\hb}\xi\rangle_{\hb}+YXuY\xi-XYuY\xi+YuYX\xi-YuXY\xi)\\
 &=\int_{\hb}y^a(-\langle\nabla_{\hb}Xu,\nabla_{\hb}\xi\rangle_{\hb}-TuY\xi-YuT\xi)\\
 &=\int_{\hb}y^a(-\langle\nabla_{\hb}Xu,\nabla_{\hb}\xi\rangle_{\hb}+2TYu\xi)
\end{align*}
(recall that $TX=XT$ and $TY=YT$). The proof of \eqref{3-2} is similar.
\end{proof}

Next result gives the first part of the inequality \eqref{1-7}. The proof is inspired by some computations in \cite{Fa,FaScVa,SiVa,StZu1,StZu2}.

\begin{theorem}\label{thm6}
Under the hypothesis of Theorem~\ref{thm1}, we have
 \begin{multline}\label{3-3}
  \int_{\hb}y^a|\nablah  u|^2|\nabla_{\hb}\phi|^2\\
  \geq \int_{\Rfp}y^a\left(|Hu|^2-\langle(Hu)^2\nu_{x,y},\nu_{x,y}\rangle_{\h}-2(TYuXu-TXuYu)\right)\phi^2.
 \end{multline}
\end{theorem}

\begin{proof}
Let us consider $\xi=|\nablah u|\phi$ as a test function in \eqref{1-5}. Thanks to \eqref{gradient_bound} and Lemma~\ref{lemma1} (see also \cite[Lemma~7]{SiVa}), it is possible to use here such a test function. We deduce that
\begin{equation}\label{3-4}
\int_{\hb}y^a(|\nabla_{\h}(|\nablah u|\phi)|^2+|\partial_y(|\nablah u|\phi)|^2)-\int_{\h}f'(u)|\nablah u|^2\phi^2\geq 0.
\end{equation}

The first term can be computed in the same way as in \cite[Theorem~1.3]{FeVa08}. We find that, in $\Rfp$,
\begin{equation}
|\nablah(|\nablah u|\phi)|^2=|\nablah u|^2|\nablah \phi|^2+\phi^2\langle(Hu)^2\nu_{x,y} ,\nu_{x,y}\rangle_\h +2\langle Hu\nablah \phi ,\nablah u\rangle_\h\phi.
\end{equation}
By exploiting Lemma~\ref{lemma2} with $w=|\nablah u|$, we obtain that $\nabla_{\hb}(|\nablah u|\phi)=0$ almost everywhere outside $\Rfp$. Analogously, using Lemma~\ref{lemma2} with $w=Xu$ or $w=Yu$, we conclude that $\nablah Xu=\nablah Yu=0$ almost everywhere outside $\Rfp$. Thus, \eqref{3-4} is equivalent to
\begin{equation}\label{3-5}
\begin{aligned}
 0\leq &\int_{\Rfp}y^a(|\nablah u|^2|\nablah \phi|^2+\phi^2\langle(Hu)^2\nu_{x,y} ,\nu_{x,y}\rangle_\h+
 2\langle  Hu\nablah \phi ,\nablah u\rangle_\h\phi) \\
 &+\int_{\Rfp}y^a|\partial_y(|\nablah u|\phi)|^2-\int_{\h}f'(u)|\nablah u|^2\phi^2.
\end{aligned}
\end{equation}

Let us now compute the last term. First, note that
\begin{equation}\label{3-6} 
\begin{aligned}
 \int_{\h}f'(u)(Xu)^2\phi^2 &= \int_{\h}X(f(u))(Xu\phi^2) \\
 &= -\int_{\h}f(u)X(Xu\phi^2).
\end{aligned}
\end{equation}
Let $\xi$ be as in the previous lemma. By the weak solution notion \eqref{1-3} and the previous lemma, we deduce that
\begin{equation}\label{3-7}
-\int_{\h}f(u)X\xi=\int_{\hb}y^a\left[\langle\nabla_{\hb}Xu,\nabla_{\hb}\xi\rangle_{\hb}-2TYu\xi\right].
\end{equation}
A density argument (see Lemma~3.4 and Theorem~2.4 in \cite{ChSe}), implies that \eqref{3-7} holds for $\xi=-Xu\phi^2$, where $\phi$ is as in statement of Theorem~\ref{thm1}. Therefore
\begin{equation}\label{3-8}
-\int_{\h}f(u)X(Xu\phi^2)=\int_{\hb}y^a\left[\langle\nabla_{\hb}Xu,\nabla_{\hb}(Xu\phi^2)\rangle_{\hb}-2TYuXu\phi^2\right].
\end{equation}
Similarly, we have
\begin{equation}\label{3-9}
 -\int_{\h}f(u)Y(Yu\phi^2)=\int_{\hb}y^a\left[\langle\nabla_{\hb}Yu,\nabla_{\hb}(Yu\phi^2)\rangle_{\hb}+2TXuYu\phi^2)\right].
\end{equation}
Then, by \eqref{3-6} and then summing term by term in \eqref{3-8} and \eqref{3-9}, we see that
\begin{equation*}
 \begin{aligned}
 \int_{\h}f'(u)|\nablah u|^2 \phi^2 =&\int_{\h}f'(u)\left[(Xu)^2+ (Yu)^2\right]\phi^2\\ 
 =&\int_{\hb}y^a(|\nablah Xu|^2+|\nablah Yu|^2)\phi^2\\
 &+\int_{\hb}y^a(\langle\nablah Xu,\nablah(\phi^2)\rangle_{\h}Xu+\langle\nablah Yu,\nablah(\phi^2)\rangle_{\h}Yu)\\
 &+2\int_{\hb}y^a(TXuYu-TYuXu)\phi^2\\
 &+\int_{\hb}y^a\left[\partial_yXu\partial_y(Xu\phi^2)+\partial_yYu\partial_y(Yu\phi^2)\right].
 \end{aligned}
\end{equation*}

Putting this in \eqref{3-5} we deduce, after a rearrangement, that
\begin{equation}\label{3-10}
 \begin{aligned}
  0\leq &\int_{\Rfp}y^a|\nablah u|^2|\nablah \phi|^2 \\
  &+\int_{\Rfp}y^a\left[-|\nablah Xu|^2-|\nablah Yu|^2+\langle(Hu)^2\nu_{x,y},\nu_{x,y}\rangle_{\h}\right. \\
  &\left.+2(TYuXu-TXuYu)\right]\phi^2 \\
  &+\int_{\Rfp}y^a|\partial_y(|\nablah u|
  \phi)|^2-\int_{\hb}y^a\left[\partial_yXu\partial_y(Xu\phi^2)+\partial_yYu\partial_y(Yu\phi^2)\right];
 \end{aligned}
\end{equation}
here we used the fact that: 
\[
2\langle  Hu\nablah \phi ,\nablah u\rangle_\h\phi-\langle\nablah Xu,\nablah(\phi^2)\rangle_{\h}Xu-\langle\nablah Yu,\nablah(\phi^2)\rangle_{\h}Yu=0.
\]

Finally, developing some calculations for the last terms in \eqref{3-10}, we conclude that
\begin{equation*}
 \begin{aligned}
 \int_{\Rfp}y^a|\partial_y &(|\nablah u|\phi)|^2-\int_{\hb}y^a\left[\partial_yXu\partial_y
 (Xu\phi^2)  +\partial_yYu\partial_y(Yu\phi^2)\right]\\
 =&\int_{\hb}y^a|\nablah u|^2(\partial_y\phi)^2+\int_{\Rfp}y^a[(\partial_y|\nablah u|\phi)^2
 +2| \nablah u|\partial_y|\nablah u|\phi\partial_y\phi\\
 &-|\partial_y\nablah u|^2\phi^2-\frac{1}{2}\partial_y|\nablah u|^2\partial_y(\phi^2)]\\
 =&\int_{\hb}y^a|\nablah u|^2(\partial_y\phi)^2+\int_{\Rfp}y^a[(\partial_y|\nablah u|)^2
 -| \partial_y\nablah u|^2]\phi^2\\
 \leq&\int_{\hb}y^a|\nablah u|^2(\partial_y\phi)^2.
 \end{aligned}
\end{equation*}
For the last inequality, note that, on $\Rfp$,
\begin{equation*}
(\partial_y|\nablah u|)^2=\left|\frac{\nablah u\cdot\nablah \partial_{y}u}{|\nablah u|}\right|^2\leq|\partial_y\nablah u|^2.
\end{equation*}
This and \eqref{3-10} complete the proof.
\end{proof}

\subsection{Geometrical computations}

To obtain the second part of \eqref{1-7}, it is necessary a geometric analysis of the level sets of $u$ at non-degenerate points $P$ where $\{\nablah u\neq 0\}$ (recall the smooth manifold $L$, defined in \eqref{1-6}). We omit the details and instead refer the reader to \cite{ArFe07,ArFe08} and \cite[Section~2]{FeVa08}. 

\begin{lemma}\label{lemma4}
On the smooth manifold $L$ we have 
 \begin{equation}
 |Hu|^2-\langle(Hu)^2\nu_{x,y},\nu_{x,y}\rangle_{\h}=|\nablah u|^2\left[h_{x,y}^2+\left(p_{x,y}+\frac{\langle 
 Huv_{x,y}, \nu_{x,y}\rangle_{\h}}{|\nablah u|}\right)^2\right]
 \end{equation}
and
 \begin{equation}
 TYuXu-TXuYu=-|\nablah u|^2\langle T\nu_{x,y},v_{x,y}\rangle_{\h}.
 \end{equation}
\end{lemma}

\begin{proof}[Proof of Theorem~\ref{thm1}]
Finally, the form of the geometric inequality given in \eqref{1-7} is a consequence of Theorem~\ref{thm6} and the previous lemma.
\end{proof}


\section{Applications to entire stable solutions}

\subsection{Proof of Theorem~\ref{thm2}}

The strategy for proving Theorem~\ref{thm2} is to test the geometric formula of Theorem~\ref{thm1} against an appropriate capacity-type function to make the left-hand side vanish. This would give that the curvature of the level sets for fixed $y>0$ vanishes.

For this, given $x=(x_1,x_2,x_3)\in\h$ we define its gauge norm as
\begin{equation}\label{4-1}
|x|_\h=\left((x_1^2+x_2^2)^2+x_3^2\right)^{1/4}.
\end{equation} 
We also use the notation $Z:=(x,y)$ for points in $\hb$ and define the norm
\begin{equation}\label{4-2}
|Z|_{\hb}:=\left(|x|_{\h}^2+y^2\right)^{1/2}
\end{equation}
(recall that $\hb=\h\times\R_+$). Analogously, we denote the ball centered at 0 of radius $R$ by
\[
B(0,R)=\{Z\in\hb \text{ s.t. } |Z|_{\hb}<R\}.
\]
and, given $r_1\leq r_2$, the semi-annulus by
\[\A_{r_1,r_2}:=\{Z\in\R^4_+ \text{ s.t. } |Z|_{\hb}\in[r_1,r_2]\}.\]

\begin{lemma}\label{lemma5}
Let $g\in L^\infty_{\mathrm{loc}}(\R^4_+,[0,+\infty))$ and let $q>0$. Let also, for any $\tau>0$,
\begin{equation}\label{4-3}
\eta(\tau)=\int_{B(0,\tau)}\! g(Z)\,\di Z.
\end{equation}
Then, for every $0<r<R$,
\[\int_{\A_{r,R}}\! \frac{g(Z)}{|Z|_{\hb}^q} \,\di Z \leq q\int_r^R \! \frac{\eta(\tau)}{\tau^{q+1}}\,\di \tau+\frac{\eta(R)}{R^q}.\]
\end{lemma}

\begin{proof}
By changing order of integration,
 \begin{align*}
 \int_{\A_{r,R}}\! &\frac{g(Z)}{|Z|_{\hb}^q} \,\di Z\\
 &=q\int_{\A_{r,R}}\! \left(\int_{|Z|_{\hb}}^R \frac{g(Z)}{\tau^{q+1}} \,\di\tau\right)\,      
 \di Z+\frac{1}{R^q}\int_{\A_{r,R}}\! g(Z)\,\di Z\\
 &\leq q\int_r^R\! \left(\int_{B(0,\tau)} \frac{g(Z)}{\tau^{q+1}} \,\di Z\right)\,      
 \di \tau+\frac{\eta(R)}{R^q}\\
 &\leq q\int_r^R \! \frac{\eta(\tau)}{\tau^{q+1}}\,\di \tau+\frac{\eta(R)}{R^q}.
 \end{align*}
\end{proof}

\begin{proof}[Proof of Theorem~\ref{thm2}]
Given $Z=(x,y)\in\hb$, let us consider the function 
\[
g(Z)=4y^a|\nablah u(Z)|^2(x_1^2+x_2^2+y^2)
\]
(recall that $x=(x_1,x_2,x_3)$). Then, the function $\eta$ defined in \eqref{1-8} is consistent with the notation in \eqref{4-3}. Moreover, by \eqref{1-10} and the previous lemma,
\begin{equation}\label{4-4}
\liminf_{R\to+\infty}\frac{1}{(\log R)^2}\int_{\A_{\sqrt{R},R}}\!\frac{g(Z)}{|Z|_{\hb}^4}\, \di Z=0.
\end{equation}

Now, we define for all $R>1$ the test function
\begin{equation*}
\phi_R(Z)=\left\{
 \begin{aligned}
 & 1 &&\text{if } |Z|_{\hb}\leq\sqrt{R},\\
 & \frac{2\log \left(\frac{R}{|Z|_{\hb}}\right)}{\log R} && \text{if } \sqrt{R}<|Z|_{\hb}<R,\\
 & 0 &&\text{if } |Z|_{\hb}\geq R, 
 \end{aligned}
 \right.
\end{equation*}
and we observe that
\begin{equation*}
\partial\phi_R =-\frac{2}{\log R}|Z|_{\hb}^{-1}\partial(|Z|_{\hb}),
\end{equation*}
where $\partial$ can be any of the operators $X$, $Y$ or $\partial_{y}$. It is straightforward to verify that
\begin{align*}
X(|Z|_{\hb})&=|Z|_{\hb}^{-1}|x|_{\h}^{-2}[x_1(x_1^2+x_2^2)+x_2x_3],\\
Y(|Z|_{\hb})&=|Z|_{\hb}^{-1}|x|_{\h}^{-2}[x_2(x_1^2+x_2^2)-x_1x_3],\\
\partial_y(|Z|_{\hb})&=|Z|_{\hb}^{-1}y.
\end{align*}
Therefore, for $Z\in\A_{\sqrt{R},R}$,
\begin{align*}
|\nabla_{\hb}\phi_R(Z)|^2 &=(X\phi_R)^2+(Y\phi_R)^2+(\partial_y\phi_R)^2\\
&=\frac{4}{(\log R)^2}|Z|_{\hb}^{-2}[X(|Z|_{\hb})^2+Y(|Z|_{\hb})^2+\partial_y(|Z|_{\hb})^2]\\
&=\frac{4}{(\log R)^2}|Z|_{\hb}^{-4}(x_1^2+x_2^2+y^2).
\end{align*}
Thus, plugging $\phi_R$ inside the geometric inequality of Theorem~\ref{thm1}, we deduce that
\begin{gather*}
\int_{B(0,\sqrt{R})^+\cap\Rfp}y^a|\nablah u|^2\left[h_{x,y}^2+\left(p_{x,y}+\frac{\langle Huv_{x,y},\nu_{x,y}\rangle_{\h}}{|\nablah  u|}\right)^2+2\langle T\nu_{x,y},v_{x,y}\rangle_{\h}\right]\\
=\frac{1}{(\log R)^2}\int_{\A_{\sqrt{R},R}}\! \frac{g(Z)}{|Z|_{\hb}^4} \,\di Z,
\end{gather*}
for all $R>1$. 

Theorem~\ref{thm2} follows from the previous identity together \eqref{1-9} and \eqref{4-4}.
\end{proof}

\subsection{Proof of Theorem~\ref{thm3}}
Before proving Theorem~\ref{thm3}, let us state the following regularity result for the extension with the Poisson kernel (see \eqref{2-4}). Recall that $C^\alpha(\h)$ denotes the set of H\"older continuous functions with respect to the norm $\rho$ defined in \eqref{rho-norm}.
\begin{lemma}
Let $v\in C^{2,\sigma}(\h)\cap L^\infty(\h)$, $\sigma\in(0,2s)$. Then the function $u(\cdot,y)=v\ast P_\h(\cdot,y)$, defined in \eqref{2-4}, satisfies 
\[
u\in C^{0,\sigma}(\overline{\hb}).
\] 
\end{lemma}
\begin{proof}
Fix $\sigma \in (0,1]$ and let $v\in C^{2,\sigma}(\h)\cap L^\infty(\h)$. Recall the following homogeneity property of $h=h(t,x)$ (the fundamental solution of $\Delta_{\h}+\partial/\partial t$):
\begin{equation}
h(r^2t,\delta_r(x))=r^{-4}h(t,x) \quad \text{for all } (t,x)\in (0,+\infty)\times \h
\end{equation}
(see (3.2) in \cite{Fo75}), where $\delta_r$ is the family of dilatations defined in \eqref{2-1a}.

We have, for all $\xi\in\h$ and $y>0$,
\begin{align*}
P_\h(\delta_y(\xi),y)&= C_a y^{1-a}\int_0^\infty \! t^{(a-3)/2}\e^{-\frac{y^2}{4t}}h(t,\delta_y(\xi))\,\di t\\
&=C_a \int_0^\infty \! t^{(a-3)/2}\e^{-\frac{1}{4t}}h(y^2t,\delta_y(\xi))\,\di t\\
&=C_a y^{-4}\int_0^\infty \! t^{(a-3)/2}\e^{-\frac{1}{4t}}h(t,\xi)\,\di t\\
&=C_a y^{-4} P_{\h}(\xi,1).
\end{align*}
Then, given $(x,y)\in \hb$, we have
\begin{align*}
u(x,y)&=\int_\h \! v(x\circ\xi^{-1})P_\h(\xi,y)\,\di\xi\\
&=y^4\int_\h \! v(x\circ\delta_y(\xi)^{-1})P_\h(\delta_y(\xi),y)\,\di \xi\\
&=C_a\int_\h \! v(x\circ\delta_y(\xi)^{-1})P_\h(\xi,1)\,\di\xi.
\end{align*}

Therefore, for $(x^{(1)},y_1),(x^{(2)},y_2)\in\hb$
\begin{align*}
|u(x^{(1)},y_1)&-u(x^{(2)},y_2)|\\
         &\leq C_a\int_\h \! |v(x^{(1)}\circ\delta_{y_1}(\xi)^{-1})-v(x^{(2)}\circ\delta_{y_2}(\xi)^{-1})|P_\h(\xi,1)\,\di \xi\\
         &\leq C\int_\h \! d(x^{(1)}\circ\delta_{y_1}(\xi)^{-1},x^{(2)}\circ\delta_{y_2}(\xi)^{-1})^{\sigma}P_\h(\xi,1)\,\di\xi,
\end{align*}
where $d$ is the homogeneous distance associated to the homogeneous norm in \eqref{rho-norm}. 

Using the properties of homogeneous norms in Carnot groups (see \cite[Section~5.1]{BoLaUg}), we deduce that
\[
d(x^{(1)}\circ\delta_{y_1}(\xi)^{-1},x^{(2)}\circ\delta_{y_2}(\xi)^{-1})\leq C[d(x^{(1)},x^{(2)})+|y_1-y_2|\rho(\xi)].
\]
Putting this in the previous inequality, we find that
\begin{equation}
|u(x^{(1)},y_1)-u(x^{(2)},y_2)|\leq C\left[d(x^{(1)},x^{(2)})^\sigma+|y_1-y_2|^\sigma\int_\h \! \rho(\xi)^\sigma P_\h(\xi,1)\,\di\xi\right].
\end{equation}
The integral in this expression is finite because the function $P_\h(\xi,1):\h\rightarrow\R$ is bounded around the origin, $\sigma\in(0,2s)$ and
\[
|P_\h(\xi,1)|\leq C\rho(y)^{-2s-4}
\] 
for large $\rho$ (see Remark~4.5 in \cite{FeFr} and (1.73) in \cite{FoSt}). We conclude that $u\in C^{0,\sigma}(\overline{\hb})$. 
\end{proof}

\begin{proof}[Proof of Theorem~\ref{thm3}]
Let $v$ be a bounded stable solution of \eqref{1-1}. We select the extension $u(\cdot,y)=v(\cdot)\ast P_\h(\cdot,y)$ by the Poisson kernel in \eqref{2-3}, that is,
\begin{equation}\label{4-5}
u(x,y)=\int_\h\! v(z)P_{\h}(z^{-1}\circ x,y)\,\di z.
\end{equation}

Let us check that $u$ satisfies the hypothesis of Theorem~\ref{thm1}. Indeed, since $P_\h(\cdot,y)\in L^1(\h)$ and $v\in L^\infty(\h)$, $u$ is well defined and bounded. Moreover, by a regularity property (see Proposition~4.3 in \cite{FeFr}),
\begin{equation}
u\in W^{1,2}_{\hb}(B_R^+;y^a\di x\di y) \quad \text{for all } R>0, 
\end{equation} 
which implies \eqref{regularity}. Therefore, $u$ is a stable weak solution (see Remark~\ref{rmk1}) of \eqref{1-2}. On the other hand, gradient bound condition \eqref{gradient_bound} follows by the following regularity argument: Given a multi-index $I$, the derivatives of $P_{\h}$ in $\h$ have the decay (see Remark~4.5 in \cite{FeFr})
\[
|\partial^I P_\h(x,y)|\leq C\rho^{-2s-4-|I|}
\]
for large $\rho=\rho(x)$. Thus, if we take ``$\h$-derivatives" in \eqref{4-5} and then use a similar argument to that in the proof of the previous lemma, we see that $u$ and its second order derivatives in $\h$ are continuous up to the boundary in $\hb$. 

Therefore, $u$ satisfies the hypothesis of Theorem~\ref{thm2}, and the level sets of $u$ intersected with $L$ (recall \eqref{1-6}) are minimal surfaces in the Heisenberg group. Theorem~\ref{thm3} follows by taking $y\to 0^+$.
 
\end{proof}


\bigskip
\noindent {\bf Acknowledgment.} L.F. L\'opez was supported by a doctoral grant of CONICYT (Chile) and {\em Fondo Basal} CMM.


\providecommand{\bysame}{\leavevmode\hbox to3em{\hrulefill}\thinspace}
\providecommand{\MR}{\relax\ifhmode\unskip\space\fi MR }
\providecommand{\MRhref}[2]{%
  \href{http://www.ams.org/mathscinet-getitem?mr=#1}{#2}
}
\providecommand{\href}[2]{#2}


\begin{thebibliography}{10}

\bibitem{ArFe07}
N.~Arcozzi and F.~Ferrari, \emph{{Minimal surfaces with isolated
  singularities}}, Math. Z. \textbf{256} (2007), no.~3, 661--684.

\bibitem{ArFe08}
\bysame, \emph{{The Hessian of the distance from a surface in the Heisenberg
  group}}, Ann. Acad. Sci. Fenn. Math. \textbf{33} (2008), no.~1, 35–63,
  http://mathstat.helsinki.fi/Annales/Vol33/vol33.html.

\bibitem{Be}
J.~Bertoin, \emph{{L\'evy Processes}}, Cambridge Tracts in Math., vol. 121,
  Cambridge Univ. Press, Cambridge, 1996.

\bibitem{BoLaUg}
A.~Bonfiglioli, E.~Lanconelli, and F.~Uguzzoni, \emph{{Stratified Lie groups
  and potential theory for their sub-Laplacians}}, Springer Monographs in
  Mathematics, Springer, Berlin, 2007.

\bibitem{CaCap}
X.~Cabré and A.~Capella, \emph{{Regularity of radial minimizers and extremal
  solutions of semilinear elliptic equations}}, J. Funct. Anal. \textbf{238}
  (2006), no.~2, 709--733.

\bibitem{CaSi07}
X.~Cabré and Y.~Sire, \emph{{Semilinear equations with fractional
  Laplacians}}, in preparation, 2007.

\bibitem{CaSo}
X.~Cabré and J.~Sol\`a-Morales, \emph{{Layer solutions in a half-space for
  boundary reactions}}, Comm. Pure Appl. Math. \textbf{58} (2005), no.~12,
  1678--1732.

\bibitem{CaRoSi}
L.~Caffarelli, J.-M. Roquejoffre, and Y.~Sire, \emph{{Variational problems with
  free boundaries for the fractional Laplacian}}, J. Eur. Math. Soc.
  \textbf{12} (2010), 1151--1179.

\bibitem{CaSaSi}
L.~Caffarelli, S.~Salsa, and L.~Silvestre, \emph{{Regularity estimates for the
  solution and the free boundary of the obstacle problem for the fractional
  Laplacian}}, Invent. Math. \textbf{171} (2008), no.~2, 425--461.

\bibitem{CaSi}
L.~Caffarelli and L.~Silvestre, \emph{{An extension problem related to the
  fractional Laplacian}}, Comm. Partial Differential Equations \textbf{32}
  (2007), no.~7--9, 1245--1260.

\bibitem{CDPT}
L.~Capogna, D.~Danielli, S.~Pauls, and J.~Tyson, \emph{{An introduction to the
  Heisenberg group and the sub- Riemannian isoperimetric problem}}, Progress in
  Mathematics, vol. 259, Birkh\"auser Verlag, Basel, 2007.

\bibitem{ChHwMa}
J.-H. Cheng, J.-F. Hwang, A.~Malchiodi, and P.~Yang, \emph{{Minimal surfaces in
  pseudohermitian geometry}}, Ann. Sc. Norm. Super. Pisa Cl. Sci. (5)
  \textbf{4} (2005), no.~1, 129--177.

\bibitem{ChSe}
V.~Chiad\`o and F.~Serra~Cassano, \emph{{Relaxation of degenerate variational
  integrals}}, Nonlinear Anal. \textbf{22} (1994), no.~4, 409--424.

\bibitem{CoTa}
R.~Cont and P.~Tankov, \emph{{Financial Modelling with Jump Processes}},
  Chapman \& Hall/CRC Financ. Math. Ser., Chapman \& Hall/CRC, Boca Raton, FL,
  2004.

\bibitem{DeGi}
E.~De~Giorgi, \emph{{Convergence problems for functionals and operators}},
  Proceedings of the International Meeting on Recent Methods in Nonlinear
  Analysis (Rome, 1978), Pitagora, Bologna, 1979, pp.~131--188.

\bibitem{DuLi}
G.~Duvaut and J.-L. Lions, \emph{{Inequalities in Mechanics and Physics}},
  Springer-Verlag, Berlin, 1976, translated from the French by C.W. John, in
  Grundlehren Math. Wiss. 219.

\bibitem{Fa}
A.~Farina, \emph{{Propri\'{e}t\'{e}s qualitatives de solutions d'\'{e}quations
  et syst\`{e}mes d'\'{e}quations non-lin\'{e}aires}}, Habilitation \`{a}
  diriger des recherches, Paris VI, 2002.

\bibitem{FaScVa}
A.~Farina, B.~Sciunzi, and E.~Valdinoci, \emph{{Bernstein and De Giorgi type
  problems: new results via a geometric approach}}, Ann. Sc. Norm. Super. Pisa
  Cl. Sci. (5) \textbf{7} (2008), no.~4, 741--791.

\bibitem{FeFr}
F.~Ferrari and B.~Franchi, \emph{{Harnack inequality for fractional
  sub-Laplacians in Carnot groups}}, arXiv:1206.0885v4 [math.AP], 2012.

\bibitem{FeVa08}
F.~Ferrari and E.~Valdinoci, \emph{{A geometric inequality in the Heisenberg
  group and its applications to stable solutions of semilinear problems}},
  Math. Ann. \textbf{343} (2009), no.~2, 351--370.

\bibitem{Fo75}
G.-B. Folland, \emph{{Subelliptic estimates and function spaces on nilpotent
  Lie groups}}, Ark. Mat. \textbf{13} (1975), no.~2, 161--207.

\bibitem{FoSt}
G.-B. Folland and E.-M. Stein, \emph{{Hardy spaces on homogeneous groups}},
  Mathematical Notes, vol.~28, Princeton University Press, Princeton, N.J.,
  1982.

\bibitem{FrLa}
B.~Franchi and E.~Lanconelli, \emph{{H\"older regularity theorem for a class of
  linear nonuniformly elliptic operators with measurable coefficients}}, Ann.
  Sc. Norm. Super. Pisa Cl. Sci. (4) \textbf{10} (1983), no.~4, 523--541.

\bibitem{FrSe}
B.~Franchi and R.~Serapioni, \emph{{Pointwise estimates for a class of strongly
  degenerate elliptic operators: a geometrical approach}}, Ann. Sc. Norm.
  Super. Pisa Cl. Sci. (4) \textbf{14} (1987), no.~4, 527--568.

\bibitem{GaDaNh}
N.~Garofalo, D.~Danielli, and D.-M. Nhieu, \emph{{Notion of convexity in Carnot
  groups}}, Comm. Anal. Geom. \textbf{11} (2003), 263--341.

\bibitem{Gi84}
E.~Giusti, \emph{{Minimal Surfaces and Functions of Bounded Variation}},
  Monogr. Math., vol.~80, Birkh\"auser-Verlag, Basel, 1984.

\bibitem{Lu92}
G.~Lu, \emph{{Weighted Poincar\'e and Sobolev inequalities for vector fields
  satisfying H\"ormander’s condition and applications}}, Rev. Mat.
  Iberoamericana \textbf{8} (1992), no.~3, 367--439.

\bibitem{Pa04}
S.~Pauls, \emph{{Minimal surfaces in the Heisenberg group}}, Geom. Dedicata
  \textbf{104} (2004), 201--231.

\bibitem{SiVa}
Y.~Sire and E.~Valdinoci, \emph{{Fractional Laplacian phase transitions and
  boundary reactions: a geometric inequality and a symmetry result}}, J. Funct.
  Anal. \textbf{256} (2009), no.~6, 1842--1864.

\bibitem{StZu1}
P.~Sternberg and K.~Zumbrun, \emph{{A Poincaré inequality with applications to
  volume-constrained areaminimizing surfaces}}, J. Reine Angew. Math.
  \textbf{503} (1998), 63--85.

\bibitem{StZu2}
\bysame, \emph{{Connectivity of phase boundaries in strictly convex domains}},
  Arch. Rational Mech. Anal. \textbf{141} (1998), no.~4, 375–400.

\end{thebibliography}
\end{document}